\pdfoutput=1
\documentclass[a4paper,12pt,reqno]{amsart}
\usepackage{amsmath, amsthm, amsfonts, amssymb, latexsym,mathtools}
\usepackage[usenames, dvipsnames]{color}
\usepackage{tikz,pgfplots}
\usepackage{cite}
\usepackage{scalerel,stackengine}
\usepackage{standalone}
\usepackage{enumerate}
\usepackage{hyphenat}
\usepackage{fullpage}

\numberwithin{equation}{section}

\newtheorem{thm}[equation]{Theorem}
\newtheorem{prop}[equation]{Proposition}
\newtheorem{lemma}[equation]{Lemma}
\newtheorem{cor}[equation]{Corollary}
\newtheorem{dfn}[equation]{Definition}

\newtheorem{clm*}{Claim}

\theoremstyle{remark}
\newtheorem{rmk}[equation]{Remark}
\newtheorem{example}[equation]{Example}

\newcommand {\iso} {\mathrm{ISO}(X)}
\newcommand {\isom} {\mathrm{ISO}_{\m}(X)}
\newcommand {\fix} [1]{\mathrm{Fix}(#1)}
\newcommand{\R}{\mathcal{R}}
\newcommand{\D}[3]{D_{#1} (#2,#3)}
\newcommand{\del}{\delta_{\epsilon}}
\newcommand{\x}{x_{\epsilon}}
\newcommand{\re}{r_{\epsilon}}

\newcommand{\Lam}{\Lambda_{\epsilon}}
\newcommand{\optgeo}[2]{\mathrm{OptGeo}(#1,#2)}
\newcommand{\geo}{\mathrm{Geo}}
\newcommand{\rcd}[2]{\mathsf{RCD}^* (#1,#2)}
\newcommand{\RCD}{\mathsf{RCD}^*}
\newcommand{\cd}[2]{\mathsf{CD}^* (#1,#2)}
\newcommand{\CD}{\mathsf{CD}^*}
\newcommand{\mcp}[2]{\mathsf{MCP}(#1,#2)}
\newcommand{\MCP}{\mathsf{MCP}}
\newcommand{\m}{\mathfrak{m}}
\newcommand{\dis}{\mathrm{d}}
\newcommand{\p}{\mathcal{P}_2}
\newcommand{\e}{\mathrm{e}}
\newcommand{\supp}{\mathrm{supp}}
\newcommand{\RR}{\mathbb R}
\newcommand{\N}{\mathbb N}
\newcommand{\mms}{\mathsf{mms}}
\newcommand{\ssp}{\mathsf{ssp}}
\newcommand{\pGH}{\mathsf{pGH}}
\newcommand{\peGH}{\mathsf{peGH}}
\newcommand{\eGH}{\mathsf{eGH}}
\newcommand{\mGH}{\mathsf{mGH}}
\newcommand{\Dn}{\mathsf{D}^n}
\newcommand{\Ln}{\mathsf{L}^n}
\newcommand{\rest}[2]{\mathsf{rest}_{#1}^{#2}}
\newcommand{\nt}{\mathfrak{n}_t}
\newcommand{\G}{\mathsf{G}}
\newcommand{\conditiona}{\textbf{($\mathsf{a}$)} }
\mathchardef\-="2D

\stackMath
\newcommand\reallywidehat[1]{%
\savestack{\tmpbox}{\stretchto{%
  \scaleto{%
    \scalerel*[\widthof{\ensuremath{#1}}]{\kern-.6pt\bigwedge\kern-.6pt}%
    {\rule[-\textheight/2]{1ex}{\textheight}}
  }{\textheight}%
}{0.5ex}}%
\stackon[1pt]{#1}{\tmpbox}%
}

\DeclareFontFamily{U}{matha}{\hyphenchar\font45}
\DeclareFontShape{U}{matha}{m}{n}{
  <-6> matha5 <6-7> matha6 <7-8> matha7
  <8-9> matha8 <9-10> matha9
  <10-12> matha10 <12-> matha12
  }{}
\DeclareSymbolFont{matha}{U}{matha}{m}{n}
\DeclareMathSymbol{\Lt}{3}{matha}{"CE}
\begin{document}
\title{The isometry group of an $\RCD$-space is Lie}
\author{Gerardo Sosa}
\thanks{Max Planck Institute for Mathematics in the Sciences,
Leipzig 04103, Germany. Email: \textsf{gsosa@mis.mpg.de.}}  
\begin{abstract}
We give necessary and sufficient conditions that show that both the group of isometries and the group of measure-preserving isometries are Lie groups for a large class of metric measure spaces. In addition we study, among other examples, whether spaces having a generalized lower Ricci curvature bound fulfill these requirements. The conditions are satisfied by $\RCD$-spaces and, under extra assumptions, by $\mathsf{CD}$-spaces, $\CD$-spaces, and $\MCP$-spaces.  However, we show that the $\MCP$-condition by itself is not enough to guarantee a smooth behavior of these automorphism groups.
\end{abstract}
\maketitle

\section{Introduction}
We give necessary and sufficient conditions to assure that the group of isometries and the group of measure-preserving isometries are Lie groups for a certain class of metric measure spaces. Additionally we analyze spaces that fulfill these assumptions, for example spaces that satisfy a particular curvature-dimension condition. Such is the case of $\RCD$-spaces, and of $\mathsf{CD}$/$\CD$-spaces, and $\MCP$-spaces satisfying mild hypotheses. More generally we show that spaces with \emph{good optimal transport} properties meet as well the hypotheses. 

In certain classes of spaces the full group of isometries, $\iso$, is known to be a Lie group. For example, Myers and Steenrod proved this fact for Riemannian Manifolds in \cite{My-St}, Fukaya and Yamaguchi for Alexandrov spaces with curvature bounded by above and by below in \cite{Fu-Ya-94, Ya}, and Cheeger, Colding, and Naber in the case of Ricci Limit spaces in \cite{Ch-Co-II,Co-Na}. In contrast, there exist metric spaces for which $\iso$ is not a Lie group, see for instance Examples \ref{ex:01} and \ref{ex:02}. 

To state the results let $(X,\dis)$ be a complete, separable metric space and assume that $\m$ is a fully supported Borel measure on $X$ which is finite on every bounded set. We call the triple $(X,\dis,\m)$ a metric measure space, $\mms$ for short, and let $\G\in\{\iso,\isom\}$ denote either the group of isometries or the group of measure-preserving isometries of $X$. The fixed point set of an isomorphism $f\in\G$ is denoted by $\fix{f}$.

\begin{thm}[$\mms$ with smooth automorphism groups]\label{th:03}
Let $(X,\dis,\m)$ be a locally compact $\mms$ where every closed ball coincides with the closure of its respective open ball. Assume that $X$ has $\m$-$a.e.$ Euclidean tangent cones.Then $\G$ is a Lie group if and only if:

\emph{\conditiona} There exist $x\in X$ and constants $0<s$, $0<\mathsf{FIX}<\m(B_s(x))$ such that for every  $(\mathbb{I}\neq) g\in \G$
\begin{equation*}
\m(\fix{g}\cap B_s(x) ) < \mathsf{FIX}.
\end{equation*}
\noindent Moreover if $\iso$ is a Lie group then $\isom$ is so as well.
\end{thm}
Furthermore, by a theorem of van Danzig and van der Waerden \cite{Da-Wa} and Lemma \ref{cl:01} we can conclude that $\iso$ and $\isom$ are compact if $X$ is compact. 

As a matter of fact Theorem \ref{th:03} remains valid when considering $\mms$ in which the tangent cones are \emph{well-behaved} $\m$-almost everywhere yet might fail to be Euclidean, see Remark \ref{rm:generalization} for the precise statement. For example, such situation  arises when tangent cones are normed spaces or Carnot groups with a uniform positive bound on the size of their subgroups of isometries. Accordingly, we are also able to study spaces with Finsler-like geometries rather than only Riemannian ones. We recall a result of relevance in this direction due to Le Donne \cite{Do}  that states that  geodesic spaces with a doubling measure that have $\m$-$a.e.$ unique tangents have $\m$-$a.e.$ Carnot groups as tangents.
 
In view of the remark the following are examples of $\mms$ satisfying condition \conditiona and the hypotheses of Theorem \ref{th:03}: Weighted Riemannian manifolds and Finsler manifolds with the Holmes-Thompson or Busemann-Hausdorff volume measure; correspondingly, Alexandrov spaces with curvature bounded below and a class of their Finsler counterpart, Busemann concave spaces \cite{Ke-16}, both endowed with the Hausdorff measure. One may ask whether these hypotheses are also granted by weaker curvature bounds. For example by \allowbreak  \textit{curvature\hyp{}dimension conditions} which use optimal mass transport theory to generalize the notion of a lower Ricci curvature bound to metric measure spaces. These conditions are variations that developed from an initial condition introduced independently by Lott and Villani and by Sturm in \cite{Lo-Vi,St-I,St-II}.
Important contributors to these developments are L. Ambrosio, K. Bacher, M. Erbar, K. Kuwada, N. Gigli, A. Mondino, T. Rajala, G. Savar\'e, and K.T. Sturm. For a historical recount one can consult for example the introductions of \cite{Mo-Na,Er-Ku-St}. 

We consider spaces satisfying the \textit{Riemannian curvature-dimension condition}, the \textit{(reduced) curvature-dimension condition}, and the \textit{measure contraction property}, and write $\RCD$, $(\CD)\mathsf{CD}$, and $\mathsf{MCP}$, for short. 
The relation between these spaces can be written as  $\RCD\-\text{spaces}\subsetneq (\mathsf{CD}\-)\CD\-\mathrm{spaces}\subsetneq \mathsf{MCP}\-\mathrm{spaces},$ where all inclusions are proper. 
 It turns out that $\RCD$-spaces have smooth isomorphism groups. 
\begin{cor}[Automorphisms of $\RCD$-spaces]\label{th:01}
Let $K \in \mathbb{R}$, $N\in [1,\infty)$, and $(X,\dis,\m)$ be an $\rcd{K}{N}$-space. Then the groups $\isom$ and $\iso$ are Lie groups. 
\end{cor}
During the completion of this manuscript Guijarro and Santos-Rodr\'iguez  proved independently in \cite{Gu-Sa}  that $\isom$ is a Lie group.  

Examples of $\RCD$-spaces are Alexandrov and  Ricci limit spaces with the Hausdorff measure, generalized cone constructions over $\RCD$-spaces, and limits of weighted manifolds with a lower bound on the Bakry-Emery Ricci tensor \cite{Pe,Lo-Vi,St-I,St-II,Ke}.   However, it is not known whether the class of $\RCD$-spaces is strictly bigger than that of weighted Ricci limit spaces.
Additionally, we consider spaces satisfying different curvature-dimension conditions. Recall that an $\mms$ is essentially non-branching and satisfies the  ($\mathsf{CD}$-)$\CD$-condition if and only if it satisfies the strong ($\mathsf{CD}$-)$\CD$-condition; for these results and definitions see \cite{Ra-St,Ca-Mo}.  
\begin{cor}[Automorphisms of $\mathsf{CD}$-, $\CD$-, and $\MCP$-spaces]\label{co:02}
Let $K$ and $N$ be as above. The groups $\iso$ and $\isom$ are Lie groups for strong $\mathsf{CD}$\emph{/}$\cd{K}{N}$-spaces and essentially non-branching $\mcp{K}{N}$-spaces that have $\m$-$a.e.$ Euclidean tangents.
\end{cor}
The $\mathsf{CD}$-, $\CD$-, and $\MCP$-conditions allow for non-Riemannian geometries which include, but are not restricted to, Finsler manifolds. Consistently with the remark following Theorem \ref{th:03}, the above corollary is still valid in spaces with these kind of metrics granted that the tangent cones are \emph{well-behaved}, see Remark \ref{rm:generalization}. For example, any corank $1$ Carnot group of dimension $(k + 1)$ equipped with a left-invariant measure is an essentially non-branching $\MCP$-space with unique non-Euclidean tangents by Rizzi \cite{Ri}, it follows that their automorphism groups are Lie groups. 

The above corollaries are particular examples of a larger class of spaces for which condition \conditiona holds. 

\begin{thm}\label{th:04}
Let $(X,\dis,\m)$ be a locally compact, length metric measure space. Assume that for all probability measures $\mu_0(\Lt\m),\mu_1\in\mathcal{P}_{2}(X)$ any optimal transport plan between $\mu_0$ and $\mu_1$ is induced by a map.
Then condition \conditiona is satisfied. In particular, if $X$ has $\m$-$a.e.$ unique Euclidean tangent cones $\G$ is a Lie group. 
\end{thm}
Indeed, from the results of Gigli-Rajala-Sturm \cite{Gi-Ra-St}, and Cavalletti-Mondino \cite{Ca-Mo} it's known that transport plans starting from absolutely continuous measures in $\RCD$-, strong $\mathsf{CD}$/$\CD$-, and essentially non-branching $\MCP$-spaces  are given by maps. Moreover it was proved in Mondino-Naber\cite{Mo-Na} that $\rcd{K}{N}$-spaces have $\m$-$a.e.$ unique Euclidean tangents for finite $N$.

On the other hand we show that a weak curvature-dimension condition by itself might not be restrictive enough to guarantee smooth automorphism groups.
\begin{prop}\label{co:03}
There exists an $\mcp{2}{3}$-space for which neither $\isom$ nor $\iso$ are Lie groups.
\end{prop}
We describe now the idea of the proof of the main theorem. A remarkable result of Gleason 
 and Yamabe 
in the early 1950's  asserts that a locally compact, topological group is not a Lie group if and only if every neighborhood of the identity has a non-trivial subgroup.\footnote{See Theorem \ref{th:Gl-Ya} and the remarks below.} If a group has this property we say that it has the small subgroup property, $\ssp$, for short. The strategy is to show the contrapositive statement in Theorem \ref{th:03}. Supposing that $\iso$ is not a Lie group, by using a blow up argument, we show that the assumption of the $\m
$-$a.e.$ infinitesimal Euclideanity and the $\ssp$ imply the existence of non-trivial isometries with arbitrarily big measure of their fixed point set. Moreover, we can verify that isometries with this property generate small subgroups. This is shown in Propositions \ref{pr:01} and \ref{pr:bigfix implies ssp} which contain most of the work needed for the proof. Subsequently we show that the existence of a single non-trivial isometry with a fixed point set of positive measure implicates that optimal plans are not unique, thus we can conclude Theorem \ref{th:04}. 
The use of blow-up arguments is common in the proofs done for Alexandrov and Ricci limit spaces. However the delicate point is to guarantee, relying simply on the tools at hand, a non-trivial convergence of subgroups of isometries acting on sequences of scaled spaces. For Alexandrov spaces one uses the fact that geodesics do not branch, whereas in the case of Ricci limit spaces, a crucial step depends on the  connectedness properties of the regular set. There exist examples in the setting of Theorem \ref{th:03} where these properties simply do not hold. Therefore we must give new arguments; we make use of optimal transport tools and measure properties of the regular set.



In the next section we give definitions and previous results that will be used. In Section \ref{se:mms} we find Propositions \ref{pr:01} and \ref{pr:bigfix implies ssp}. The rest of the work needed to conclude the proofs, and the proofs themselves, of Theorem \ref{th:03} and \ref{th:04}, are in Section \ref{se:pr}. At the end of the manuscript we present an example of an $\MCP$-space where $\isom$ is not a Lie group.
\section*{Acknowledgements}
The author would like to express his gratitude to J{\"u}rgen Jost for his crucial support and valuable advice. He is also very grateful to Rostislav Matveev, and Jim Portegies for their essential and friendly involvement in this work. \\ He also thanks Fernando Galaz-Garcia, Yu Kitabeppu, Martin Kell, and Tapio Rajala for, among other things, discussions that led to Remark \ref{re:01}, and Fabio Cavalletti and Andrea Mondino for bringing \cite{Ca-Mo} to his attention. Lastly, the author is thankful to Nidhi Kaihnsa, Niccol\`o Pederzani, and Ruijun Wu for commenting on earlier versions of the manuscript. This work was supported by the IMPRS program of the Max Planck Institute for Mathematics in the Sciences and partially by CONACYT. 
\section{Preliminaries}\label{se:ba}
We set notation and compile definitions and results used in the paper. The text is mainly self-contained, however, we provide references for some more elaborate definitions to maintain brevity.
\subsection{Metric measure spaces}
A metric measure space, $(X,\dis,\m)$, is a triple where
\begin{itemize}
\item[]$(X,\dis)$ is  a $complete, \; separable\; metric\; space$ and,
\item[] $\m \neq 0$ is a \textit{non-negative Borel measure finite on every bounded set}. 
\end{itemize}
We write $\mms$ for short. A pointed metric measure space, $(X,\dis,\m,x)$, is a $\mms$ together with a base point $x\in X$. In the text a geodesics is a map, $\gamma : [0,1] \to X$, such that:
\begin{equation*}
\dis(\gamma_r,\gamma_s) = (r-s)\dis(\gamma_0,\gamma_1) \qquad \text{for all } 0\leq s \leq r \leq 1
\end{equation*}
where $\gamma_t:=\gamma(t)$. We write $\geo(X)$ for the space of all geodesics on $X$ endowed with topology of uniform convergence. A metric space is called a \textit{geodesic space} if for every given pair of points $x,y\in (X,\dis)$ there exists a geodesic that joins $x$ and $y$.  For $t\in [0,1]$ define the \textit{evaluation map}, $\e_t:\geo(X) \to X$, as $\e_t(\gamma):=\gamma_t$ for $\gamma\in\geo(X)$. The \textit{restriction map}, $\rest{s}{t}:\geo(X) \to \geo(X)$, is defined as $\rest{s}{t}:= \gamma \circ f_s^t$ for $s,t\in[0,1], \gamma \in \geo(X)$ and the real function $f_s^t(x):=(t-s)x+s$.

We deal with group actions on sequences of $\mathsf{pmms}$,  and the pointed Gromov-Hausdorff ($\pGH$) and pointed equivariant Gromov-Hausdorff convergence ($\peGH$) provide canonical types of convergence in this framework. For definitions we refer to \cite{Bu-Bu-Iv,Fu,Fu-Ya-92}.
 
Any $\pGH$-limit of a sequence of scaled spaces, $(X,\frac{1}{r_i}\dis,x)\overset{\pGH}{\to}(X_x,\dis_{\infty},x_{\infty})$ for $r_i \to 0$, is called a (metric) \text{tangent cone of $X$ at $x$}. We denote the set of all tangent cones of $X$ at $x$ by $\mathsf{Tan}(X,x):=\{(X_\infty,\dis_{\infty},x_{\infty}) \text{ is a } \pGH\text{-limit as above}\}$. The existence or uniqueness of tangent cones is not guaranteed in a general setting. However, the set of points of $X$ with unique Euclidean tangent cones is called the $regular$ $set$ of $X$, written as $\R$. In detail,
\begin{align*}
\R  :=  \left \{ x \in X \, \| \, \exists k=k(x) \in \mathbb{N} \text{ such that } \mathsf{Tan}(X,x)= \{(\mathbb{R}^k,d_E,\overline{0})\}\right \}. 
\end{align*}
We say that $(X,\dis,\m)$ has \textit{$\m$-$a.e.$ Euclidean tangents} if the set of numbers $\{k(x) \in \mathbb{N} \,\| \, x \in \R,\, \mathsf{Tan}(X,x)=\{(\mathbb{R}^{k(x)},d_E,\overline{0})\}\}$ is finite and if $\m(X\setminus \R)=0$. For a fixed $\epsilon> 0 $, the $\epsilon$-$regular$ $set$, $\R_\epsilon$, is the set $\R_\epsilon:= \cup_\delta (\R)_{\epsilon, \delta}$, where for a given $\delta>0$ the set $(\R)_{\epsilon,\delta}$ is defined as all points $x\in X$ for which there exists a $k=k(x)$ such that
\begin{equation*}
\dis_\mathsf{GH} \left (B_r(x),B_r^k(\overline{0})\right ) < \epsilon \; r \qquad \text{for all } r< \delta.
\end{equation*}
Above $\dis_{\mathsf{GH}}$ is the Gromov-Hausdorff distance and $B_r^k(\overline{0}) \subset \mathbb{R}^k$ is the ball of radius $r$ around $\overline{0} \in \mathbb{R}^k$. Note that $\R = \cap_{\epsilon }\R_\epsilon$ and that for every $\epsilon >0$ the measure $\m(X\setminus \R_\epsilon)$=0 if $X$ has $\m$-$a.e.$ Euclidean tangents.

Two $\mms$ $(X_1,\dis_1,\m_1)$, $(X_2,\dis_2,\m_2)$ are isomorphic if there exists an isometry 
\begin{align}\label{def:iso}
\begin{split}
&f:\supp(\m_1) \to X_2 \qquad \mathrm{such\; that }\\
&(f)_{\#}\m_1=\m_2.
\end{split}
\end{align}
We use the word $isometry$ to make reference to usual metric isometries. In contrast, we refer to maps satisfying (\ref{def:iso}) as  \textit{measure-preserving isometries}.  Particularly, we note that an isometry is defined on the whole space $X_1$ and does not necessarily satisfy (\ref{def:iso}). By definition $(X,\dis,\m)$ is always isomorphic to $(\supp(\m),\dis,\m)$. This induces a canonical equivalence class of isometric metric measure spaces where \textbf{only the support of the measure is relevant}. \rm We assume that $\supp(\m)=X$, which is a natural restriction in the class of isomorphisms of $\mms$. We endow the groups $\iso$ and $\isom$ with the compact-open topology making them topological groups, see \cite{Ko-No} pp.46. We write in the remainder $\G\in\{\iso,\isom\}$ to denote one of these two groups. 
\begin{rmk}[Topology on $\isom$]\label{re:01}
We explain and motivate our choice of topology on $\isom$. For locally compact metric spaces it's natural to endow  $\iso$ with the compact-open topology since the structure under study is of pure metric nature. In addition, in this context, the rigidity of the isometries assures that pointwise convergence implies convergence w.r.t. the compact-open topology.\footnote{Rigorously we would have to justify the use of sequences to compare topologies. This can  be done because $\iso$ is second-countable which can be concluded from the fact that $X$ is a locally compact metric space. Consult for instance \cite{Ko-No} pp.46.}\label{footnote} 
Alternatively, on $\mms$ there is additional structure of interest, namely, the measure structure. However, as we explain below, the rigidity of the measure-preserving isometries guarantee that a reasonable choice of topology on $\isom$ coincides with the compact-open topology.

We first observe that topology that only considers the measure structure is too coarse for our purposes because it doesn't see metric properties. A logical way to proceed would be to couple a measure-wise and a metric-wise topology. However, the weakest metric convergence, the pointwise convergence, coincides with the compact-open convergence. On the other hand, in Lemma \ref{cl:01} we show that the compact-open convergence of a sequence of measure-preserving isometries, $(f_n)$, implies the weak convergence of the pushforward measures $(f_n)_\#(\m)$ in a locally compact $\mms$. 
\end{rmk}	
\subsection{Lie Groups}
Denote by $G_0$ the identity component of $G$, that is, the largest connected set containing the identity element $\mathbb{I}$. As definition we say that $G$ is a Lie group if and only if $G / G_0$ is discrete\footnote{For a comment on the cardinality of $G/G_0$ see Remark \ref{rm:car}} and the identity component $G_0$ is  a Lie group in the usual smooth sense. We know by Remark \ref{rm:car} that $G$ looks, in the worst cases, as countable copies of a smooth Lie group which do not accumulate. 
 \begin{thm}[van Dantzig and van der Waerden (1928) \cite{Da-Wa}]\label{th:Da-Wa}
Let $(X,d)$ be a connected, locally compact metric space. Then $\iso$ is locally compact with respect to the compact-open topology. Furthermore if $X$ is compact, then $\iso$ is compact.
\end{thm}
A topological group has $G$ has the $no \; small\;subgroup$ $property$ if there exists a neighborhood of the identity with no non-trivial subgroup. In this case we write $\mathsf{nssp}$ for short. Below we cite an outstanding result that characterizes Lie groups in terms of the $\mathsf{nssp}$. 
\begin{thm}[Gleason (1952) \cite{Gl}, Yamabe (1953) \cite{Ya53}]\label{th:Gl-Ya}
Let $G$ be a locally compact, topological group. Then $G$ is a Lie group if and only if it has the no small subgroups property. 
\end{thm}
\begin{rmk}
In \cite{Ya53} Yamabe generalizes Gleason's theorem to the infinite dimensional case, however, $G$ is assumed to be connected. We present an argument, due to an undisclosed Russian mathematician, which shows that we can  consider non-connected groups. 

An equivalent way of stating Theorem \ref{th:Gl-Ya} is: Assuming  the same hypothesis, then there exists an open subgroup $G'< G$ such that for every neighborhood of the identity $ U\subset G$ there exists a normal subgroup $(U\supset )K\unlhd G$ that makes $G'/K$ a Lie Group \cite{Tao}. If $G$ has the $\mathsf{nssp}$ the only small normal subgroup is $\mathbb{I}$ itself, thus making $G'$ a Lie group which by definition means that $G'/G_0$ is discrete. However, this implies that $G/G_0$ is discrete since $G/G'$ is also discrete and $G/G'=(G/G_0) \big / (G' /G_0) $. 

\end{rmk}
\begin{rmk}\label{rm:car}
We make another observation regarding the cardinality of $G / G_0$. In principle the group of components might be uncountable but  fortunately, we can also discard this behavior. Assume that $G$ is a second-countable Lie group. By definition $G / G_0$ is discrete which is equivalent to $G_0$ being open. In turn, this implies that the quotient map is open and it follows that $G / G_0$ is second-countable since by assumption $G$ is second-countable. A second-countable space is separable, and discrete separable spaces are countable.  
Finally, we recall that $\iso$ is second-countable for a locally compact, connected metric space.
\end{rmk}
\subsection{Curvature dimension conditions}\label{se:rcd}
Curvature-dimension conditions require certain convexity behavior of an entropy functional defined on the space of probability measures of an $\mms$. Different choices of entropy functional and different types of convexity conditions give rise to alternative versions of curvature-dimension conditions. However, all conditions are compatible with lower Ricci bounds in the smooth framework and are stable under pointed measured-Gromov-Hausdorff convergence. Optimal transport theory provides an appropriate framework to define these type of conditions.  

We've decided not to include the definitions of the $\RCD$- or the $\CD$-conditions since we don't use them explicitly and are rather technical. Alternatively, we present the results that show that these spaces satisfy the assumptions of Theorem \ref{th:04}. Nevertheless, let's do recall that the $\RCD$-condition condition was introduced in  \cite{Am-Gi-Sa} and further developed in \cite{Am-Gi-Mo-Ra} and \cite{Er-Ku-St} to which the reader is referred to for a comprehensive discussion on the subject. The $\RCD$-condition itself couples a curvature-dimension condition with an infinitesimal Riemannian behavior known as infinitesimally Hilbertianity. 

Let $\mathcal{P}(X)$ be the space of probability measures on $(X,d)$ and $\p (X) \subset \mathcal{P}$ the subspace of measures with finite second moments. For $\mu_0,\mu_1 \in \p(X)$ the Wasserstein squared distance is defined as 
\begin{equation}\label{eq:w-d}
W_2^2(\mu_0,\mu_1):= \inf_\sigma \int_{X\times X} \dis (x,y)^2 \mathrm{d}\sigma (x,y).
\end{equation} 
The infimum taken over all measures $\sigma \in \mathcal{P}(X\times X)$ with first and second marginals equal to $\mu_0$ and $\mu_1$ respectively. 
If there exists a measurable function $G:X \to X $ such that the measure $\sigma=(\mathbb{Id},G)_\#\mu_0$ is a minimum we call $\sigma$ an \textit{optimal map}. 
Given $\mu_0,\mu_1 \in \p(X)$ the set $\optgeo{\mu_0}{\mu_1} \subset \mathcal{P}(\geo(X))$ is defined as the set of all measures $\pi$ such that the pushforward $(\e_0,\e_1)_{\#}\pi \in \mathcal{P}(X\times X) $ realizes the minimum in (\ref{eq:w-d}). A measure $\pi\in\optgeo{\mu_0}{\mu_1}$ is called an \textit{optimal geodesic plan} and if such a measure $\pi$ is the lift of an optimal map we call it an \textit{optimal geodesic map}. Note that the existence of optimal maps is rare however, the next theorem shows their existence in $\mms$ that satisfy a curvature-dimension condition and that do not branch too much. Recall that the essentially non-branching condition says that geodesics do not branch too often, we refer to the cited articles for a precise definition.    

\begin{thm}[Existence of optimal maps. Cavalleti-Gigli-Mondino-Rajala-Sturm \cite{Gi-Ra-St,Ca-Mo}]\label{th:Gi-Ra-St}
Let $K \in \mathbb{R}$, $N\in [1,\infty)$ and $(X,\dis,\m)$ be an essentially non-branching $\mcp{K}{N}$ space. Then, for every $\mu_0(\Lt \m),\mu_1 \in \p(X)$, there exist a unique optimal geodesic plan $\pi \in \optgeo{\mu_0}{\mu_1}$. Furthermore, such $\pi$ is given by a map. In particular, there exists $\Gamma \subset \geo(X)$ with $\pi (\Gamma)=1$ such that the map $e_t: \Gamma \to X$ is injective for all $t \in [0,1)$. 
\end{thm}
We recall that essentially non-branching $\MCP$ space include: $\RCD$-spaces, essentially non-branching $\CD$-spaces, and essentially non-branching $\mathsf{CD}$-spaces.

The existence of $\m\- a.e.$ Euclidean tangents in $\RCD$-spaces was proved in \cite{Gi-Mo-Ra}, and later the uniqueness was shown by Mondino and Naber as a byproduct of the rectifiability of $\rcd{K}{N}$-spaces as metric spaces (compare with \cite{Ke-Mo} for rectifiability as $\mms$). A weaker version of their result is enough for us.
\begin{thm}[$\m$-$a.e.$ Euclidean tangents in $\rcd{K}{N}$-spaces. Mondino-Naber (2014)  \cite{Mo-Na}]\label{th:Mo-Na}
Let $K,N \in \mathbb{R}$, $N\geq 1$ and $(X,\dis,\m)$ be an $\rcd{K}{N}$ space. Then $ X$ has $\m$-$a.e.$ Euclidean (metric) tangents, i.e. $\m(X\setminus \R)=0$.
\end{thm}
In the last section we will work with Ohta's definition of the $\mcp{K}{N}$-condition. Intuitively, it requires contraction properties of measures as the $\RCD$-condition but only for starting $\delta$ measures. The condition is stable w.r.t. to $\mathsf{pmGH}$-convergence \cite{Oh-MCP}. We give the specific shape of the condition for the case we will study. 
\begin{dfn}[$\mcp{2}{3}-condition$]\label{def:MCP}
A $\mms$, $(X,\dis ,\m)$, has the $(2,3)$-measure contraction property, $\MCP(2,3)$, if 
for every point $x \in X$ and a measurable set $A \subset X$ with $0 < \m(A) < \infty$ and $A \subset B(x,\pi)$   
there exists a probability measure $\pi\in \mathcal{P}(\geo(X))$ such that $(e_0)_\# \pi = 
\delta_x$, $(e_1)_\#\pi = \m(A)^{-1}\m|_{A}$, and
\begin{equation}\label{eq:mcp}
(e_t)_\# \left(t \frac{\sin^2(t\, l(\gamma))}{\sin^2(l(\gamma))}\m(A)\mathrm{d}\pi(\gamma)\right)  \leq \mathrm{d}\m \qquad t \in [0,1].
\end{equation}
\end{dfn}
\section{Metric measure spaces with $\G$ containing small subgroups}\label{se:mms}
We study $\mms$ where  $\G$ has small subgroups. Granted that $\G$ has the small subgroup property we show the existence of many automorphisms with a large fixed point set, $\fix{f}$, and vice versa: these type of automorphisms create small subgroups. This identification is the main step needed in the prove Theorem \ref{th:03}. 

Define for $r>0$, $x \in X$, and a subgroup $\Lambda \leq§ \iso$: $$\D{\Lambda}{r}{x} := \sup_{g\in \Lambda}\sup_{y \in B_{\frac{r}{2}}(x) } \dis(y,g(y)).$$ For fixed $\Lambda$, the function $\D{\Lambda}{r}{x} $ is continuous in $r$ and $x$ as long as every closed ball in $X$ is the closure of its interior  \cite{Ch-Co-II}. This holds true when $X$ is a length space, for example. We do not assume a fully supported measure in the coming proposition.

\begin{prop}\label{pr:01}
Let $(X,\dis,\m)$ be a $\mms$ where every closed ball coincides with the closure of the open ball. Assume that $X$ has $\m$-$a.e.$ Euclidean tangents.
If $\G$ has the small subgroups property, then for every $x\in X$,  $0<s$, and $0<\xi<1$ there exists a non-trivial subgroup $\Lambda=\Lambda_{x,s,\xi} \subset\G$ such that for every $g\in\Lambda$ 
\begin{equation}
\m(\fix{g}\cap B_{s}(x))\geq \xi\; \m(B_{s}(x)).
\label{eq:resultado}
\end{equation} 
\end{prop}
\begin{proof}
We assume that $\m(B_s(x))>0$ since the inequality above trivially holds true otherwise. We argue by contradiction. The strategy is the following: assuming that inequality (\ref{eq:resultado}) doesn't hold we will find for every $\epsilon > 0$ a quadruple $(\delta_{\epsilon},\re,x_{\epsilon},\Lam) \in (\mathbb{R}^+)^2\times X\times 2^{\G}$ with the following properties:
\begin{equation}
\begin{aligned}
\bullet \;& 0<r_{\epsilon}\leq\delta_{\epsilon}  \\
\bullet \;& x_{\epsilon} \in (\mathcal{R})_{\epsilon,\delta_{\epsilon}}\\
\bullet \;& \Lam \leq \mathrm{ISO}(X) \text{ is a subgroup}\\
\bullet \;& \D{\Lam}{\re}{\x} = \frac{\re}{20} .
\end{aligned}
\label{eq:condiciones}
\end{equation}
The existence of such a family of quadruples would lead to a contradiction and thus, would prove the proposition. Indeed, observe that if for every $\epsilon > 0$ there exists a quadruple as above, then for a sequence $\epsilon_n \rightarrow 0$ there exists a subsequence  $\epsilon_n$ (denoted in the same way) such that, in the $\eGH$-sense, the scaled spaces below converge to
\[
\left( B_{\re}(\x),\frac{1}{\re}\mathrm{d},\Lam\right) \xrightarrow{\eGH}
\left( B_{1}(\vec{0})\subset \mathbb{R}^k,\mathrm{d}_E,\Lambda_{\infty}\right),
\]
where $k\in \mathbb{N}$, and $\Lambda_{\infty} \leq \mathrm{ISO}(\mathbb{R}^k)$ is a non-trivial subgroup satisfying $\D{\Lambda_{\infty}}{1}{\vec{0}}=\frac{1}{20}$. This creates the contradiction, since every non-trivial subgroup of Euclidean isometries $\mathsf H$ fulfills $\D{\mathsf H}{1}{\vec{0}}\leq \frac{1}{20}$.  

We proceed to construct a family of quadruples satisfying conditions (\ref{eq:condiciones}). Suppose that (\ref{eq:resultado}) does not hold. That is, there exist $x \in X$, $0<s$, and $0<\xi<1$ such that for every non-trivial subgroup $\mathsf{H}\subset \G$ there exists an $f\in \mathsf{H}$ where $$\m(\fix{f}\cap B_{s}(x))<\xi \; \m({B_{s}(x)}).$$ Note that necessarily$f\neq \mathbb{I}$. Take $\epsilon>0$ and choose small enough $\del \in \mathbb{R}$ so that $0<\del< s $, and
\begin{equation}
\xi \; \m(B_s(x))<\m( (\mathcal{R})_{\epsilon,\del}\cap B_s(x)).
\label{eq:regular grande}
\end{equation}
The $\m$-$a.e.$ Euclidean tangents of $X$, together with the continuity from below of the measure and the fact that $\R \subset (\R)_{\epsilon} = \cup_{\delta > 0}(\R)_{\epsilon,\delta}$ make possible the choice of such a $\del$. Indeed, since for $\delta'\leq \delta''$ it holds that  $ (\R)_{\epsilon,\delta''} \subset (\R)_{\epsilon,\delta'}$  we can write $(\R)_{\epsilon}=\cup_{n \in\mathbb{N}}(\R)_{\epsilon,1/n}$ as a countable union of sets. Now just notice that 
\begin{eqnarray*} 
 \m(B_s(x))&=&\m(B_s(x)\cap\R_{\epsilon})=\m(B_s(x)\cap (\cup_{n \in\mathbb{N}}(\R)_{\epsilon,1/n}))\\&=&\lim_{n\rightarrow \infty}\m(\cup_{j \leq n}(B_s(x)\cap (\R)_{\epsilon,1/j}))\\&=&\lim_{n\rightarrow \infty}\m(B_s(x)\cap(\R)_{\epsilon,1/n}).
 \end{eqnarray*}
  Choose $n \in \mathbb{N}$ big enough and take $\del <\min\{ s,1/n\}$. 
  
  Inequality (\ref{eq:regular grande}) combined with the $reductio\; ad\; absurdum$ assumption imply that for every non-trivial $\mathsf H \leq \iso$ there exist $f(\neq \mathbb{Id}) \in \mathsf H$ such that the set $B_s(x)\cap (\mathcal{R})_{\epsilon,\del}\setminus \fix{f}$ is not empty.

In view of the small subgroups property of G, we can find a non-trivial small subgroup
 \begin{align*}
\Lam \subset U_{\epsilon} := & \left\{ g\in\G\; \| \; \sup_{y \in B_{2s}(x)} \mathrm{d}(y,g(y)) < \frac{\del}{20}  \right \} \\ = &
\left \{ g\in\G \;\| \; g(y)\in B_{\del /20}(y) \text{ for all } y \in \overline{B_{2s}}(x)  \right\}.  
\end{align*} 
In particular, 
there exist $g(\neq \mathbb{Id})\in \Lam$ and $x_{\epsilon}\in B_s(x) $ such that
\begin{eqnarray*}
x_{\epsilon}&\in& B_s(x)\cap (\R)_{\epsilon,\del}   \setminus \fix{g} \text{ and }\\
  0&<& \dis(\x,g(\x)) <  \frac{\del}{20}.
\end{eqnarray*}
Denote by $\theta=\theta(\x) := 20\; \dis(\x,g(\x))<\del$. By construction it follows that 
\begin{equation*}
\begin{aligned}
\frac{1}{20}\theta \leq &\D{\Lam}{\theta}{x_{\epsilon}} \\
&\D{\Lam}{\del}{x_\epsilon} \leq \D{\Lam}{4s}{x} < \frac{1}{20}\del.
\end{aligned}
\end{equation*}
Finally, the continuity of $\D{\Lam}{\circ}{x_\epsilon}$ and the intermediate value theorem imply that there exists $\re \in \mathbb{R}$ such that $ \D{\Lam}{\re}{x_{\epsilon}}=\frac{1}{20}\re  $ for some $\theta \leq \re < \del$. Hence for $\epsilon > 0$ there exists a quadruple $(\delta_{\epsilon},\re,x_{\epsilon},\Lam)$ satisfying (\ref{eq:condiciones}).
\end{proof}
\begin{rmk}\label{rm:01}
Note that the logical negation to the conclusion of Proposition \ref{pr:01} is equivalent to condition \conditiona of Theorem \ref{th:03}.
\end{rmk}
Next, we see that we can generate small subgroups from the existence of automorphisms with large fixed point sets. 
\begin{prop}\label{pr:bigfix implies ssp}
Let $(X,\dis,\m)$ be a locally compact $\mms$ where every closed ball coincides with the closure of its interior. Then $\G$ has the small subgroups property if for every $x\in X$,  $0<s$, and $0<\xi'<1$ there exists a non-trivial subgroup $\Lambda=\Lambda_{x,s,\xi'} \subset\G$ such that for every $g\in\Lambda$
\begin{equation*}
\m(X \setminus \fix{g} \cap B_s(x)) \leq \xi'\;\m(B_s(x)).
\end{equation*}
\end{prop}
\begin{proof}
We give a sequence $\{\xi_N\}_{N\in \N}\subset (0,1)$ which generates, according to the hypothesis, a sequence of non-trivial subgroups  $\{\Lambda_{x,N,\xi_N}\}_{N\in \N}\leq \G$ such that $\Lambda_N \subset U_{N}(\ni \mathbb{I})$ for every $N\in\N$, where $\{U_N\}_{N\in \N}\subset \G$ is local basis of the compact-open topology at $\mathbb{I}$. Thus proving the existence of small subgroups of $\G$. 

Accordingly we fix $x\in X$, $N\in\N$, and define 
$$\xi'_{N}:=\allowbreak \m(B_N(x))^{-1} \allowbreak \inf_{y \in B_N(x)}\{\m(B_{1/N}(y)\cap B_N(x))\}.$$ We claim that $0<\xi'_N$. Indeed, choose a converging sequence\footnote{Using a subsequence if necessary.} $y_m \to y_\infty \in \overline{B}_N(x)$ such that $\liminf_{m\to\infty}\m(B_N(x))^{-1}\m(B_{1/N}(y_m)\cap B_N(x))) = \xi'_{N}$.\footnote{The existence of such subsequence is guarantee since locally compact, complete metric spaces for which the closure of open balls coincides with closed balls are proper.}  Since the measure $\m$ has full support there exists a small ball, $B_\tau(y_\infty)$, with $\m(B_\tau(y_\infty) \cap B_N(x))>0$ which is a lower bound of $\m(B_{1/N}(y_m)\cap B_N(x))$ for large enough  $m$, hence,  validating the claim. 
We take $0<\xi_N<\xi'_N$ and write $\Lambda_N:=\Lambda_{x,N,\xi_{N}}$ for the  non-trivial subgroup given by the hypothesis for the triple $(x,N,\xi_N)$. By construction we verify that
\begin{equation}\label{eq:bound fix co}
\m(X \setminus \fix{f} \cap B_N(x)) < \m(B_{1/N}(y)\cap B_N(x))
\end{equation}
for every $y \in B_N(x)$ and $f\in \Lambda_N$.

Observe now that if  $\dis(z,g(z))> 2\,t $ then $B_t(z) \cap B_R(x)\subset X \setminus \fix{f} \cap B_R(x)$ for $g\in \G$, $z\in X$, and numbers $t,R\in \RR^+$. Therefore, we conclude from \eqref{eq:bound fix co}  that for every $y \in B_N(x)$ and $f\in \Lambda_N$ we have that $d(y,f(y))<2/N$. Hence $\Lambda_N$ is contained in the neighborhood of the identity:
\begin{equation*}
U_{N}:=\left \{ g\in \G \,\|\, \dis(y,g(y))< 3/N \text{ for every } y\in \overline{B}_N(x) \right \}.
\end{equation*}
Accordingly, the proof is complete considering that the choice of $N$ was arbitrary.
\end{proof}
\section{Proof of Theorem \ref{th:03}}\label{se:pr}
We start this section with a lemma that shows that the uniqueness of optimal geodesic maps is sufficient to guarantee that non-trivial isometries have fixed point sets of measure zero. In particular, Theorem \ref{th:04} is concluded from this result and Theorem \ref{th:03} after taking into consideration that locally compact, complete length spaces are geodesic.

\begin{lemma}[Zero measure of the fixed point set]\label{lm:02}
Let $(X,\dis,\m)$ be a $\mms$ such that for every $\mu_0,\mu_1 \in \p(X) $ with $\mu_0 \Lt \m$ there exists a unique optimal geodesic plan $\pi \in \optgeo{\mu_0}{\mu_1}$. Furthermore, assume that $\pi$ is concentrated on a set of geodesics, $\Gamma \subset \geo(X)$, such that the map $\e_0:\Gamma \rightarrow X$ is injective. Let $f\neq \mathbb{I}$ be an isometry of $X$. Then $\m(\fix{f})=0 $.
\end{lemma}
\begin{proof}
We proceed by contradiction. Suppose that there exist $\mathbb{I}\neq f \in \iso$, and a set $A\subset \fix{f}$ with positive measure. Let $x\in X \setminus \fix{f}$ and define the probability measures $\mu_0 := \m(A)^{-1}\m|_{A}$ and $\mu_1:=\frac{1}{2}(\delta_{x} + \delta_{f(x)})$. We denote by $\pi \in \optgeo{\mu_0}{\mu_1}$ the unique geodesic plan between $\mu_0$ and $\mu_1$. Let $\Gamma \subset \geo(X)$ be the set where $\pi$ is concentrated and where $\e_0$ is injective. Set:
\begin{eqnarray*}
\Gamma_1 &:=&\left\{ \gamma \in \Gamma \;\;|\;\; \e_1(\gamma)=x\right\}\\
\Gamma_2 &:=& \left\{ \gamma \in \Gamma \;\;|\;\; \e_1(\gamma)=f(x)\right\}\\
A_i&:=&\e_0(\Gamma_i) \qquad i=1,2.
\end{eqnarray*}
$\Gamma_1 \;(\Gamma_2)$ is the subset of geodesics of $\Gamma$ that end in $x \;(f(x))$ and $A_i$ is the projection of $\Gamma_i$ onto the set $A$. We have that none of these sets are empty, that the measures $\pi\left(\Gamma \setminus ( \Gamma_1\cup \Gamma_2)\right)=0=\m\left(A\setminus (A_1 \cup A_2)\right) $ and that $A_1\cap A_2 = \emptyset$. The last fact is a consequence of the injectivity of $e_0$. We define now the measure $\pi'\in\mathcal{P}(\geo(X))$ as
\begin{equation}
\pi':=(\hat{f})_{\#}\pi|_{\Gamma_1}+(\reallywidehat{{f^{-1}}})_{\#}\pi|_{\Gamma_2},
\end{equation}
where the  bijection of $\geo(X)$, $\gamma\mapsto g\circ \gamma$, induced by some $g\in \iso$  is written as $\hat{g}:\geo(X)\rightarrow\geo(X)$. The measure $\pi'$ is a symmetric analog of $\pi$ but $\pi'\neq \pi$. Indeed, note that $\pi'(\hat{f}( \Gamma_1)) = 1/2 \neq 0 =\pi(\hat{f}(\Gamma_1))$ because $\hat{f} (\Gamma_1)\cap \Gamma_1= \emptyset $  by construction.  

We claim that $\pi'\in\optgeo{\mu_0}{\mu_1}$ is also a dynamical plan. This would contradict the hypothesis of the uniqueness of $\pi$ and finish the proof of the lemma. We proceed to verify the claim.

We need to show that $\pi'$ minimizes $\int_{\geo(X)} l(\gamma) d\rho$. The minimum taken over all measures $\rho \in \mathcal{P}(\geo(X))$ such that $(e_i)_{\#}\rho=\mu_i$ for $i=0,1$. We check that the pushforwards of $\pi'$ under the evaluation map are as above. For this we observe that for $g\in \iso$, $B\subset X$, and $t\in[0,1]$ 
\begin{eqnarray*}
\hat{g} \circ \e_t^{-1} (B) &=& \hat{g}\left( \left \{ \gamma \in \geo(X) \;\;|\;\; \e_t(\gamma)\in B  \right \} \right) \\
&=& \left \{ \gamma \in \geo(X) \;\;|\;\; \e_t \in g(B) \right \} = \e_t^{-1} \circ g (B),
\end{eqnarray*}
and that $\hat{g}^{-1}=\reallywidehat{g^{-1}}$. Next we compute the pushforward of $\pi'$ under $\e_t$:
\begin{eqnarray*}
(\e_t)_{\#}\pi'&=&(e_t\circ\hat{f})_{\#}\pi|_{\Gamma_1}+(e_t\circ \reallywidehat{f^{-1}})_{\#}\pi|_{\Gamma_2} \\
&=& (f)_\# (e_t)_{\#}\pi|_{\Gamma_1}+(f^{-1})_\# (e_t)_{\#}\pi|_{\Gamma_2}.
\end{eqnarray*}
Then $(\e_0)_\# \pi'= \mu_0 $ since $f|_A=\mathbb{I}|_A$. As for the other pushforward we have that $(\e_1)_\#\pi' = (f)_\#(\frac{1}{2}\delta_x) +(f^{-1})_\#(\frac{1}{2}\delta_{f(x)})=\frac{1}{2}(\delta_x+\delta_{f(x)})=\mu_1$. 
To finish we see that $\pi'\in\optgeo{\mu_0}{\mu_1}$ by  showing that the value of $\int_{\geo(X)} l(\gamma) \dis \pi'$ is the minimum of the functional.
\begin{eqnarray*}
\int_{\geo(X)} l(\gamma) \dis \pi'(\gamma)& =&\int_{\geo(X)} l(\gamma) \dis \left( (\hat{f})_{\#}\pi|_{\Gamma_1}+(\reallywidehat{f^{-1}})_{\#}\pi|_{\Gamma_2} \right)(\gamma)\\
&=& \int_{\geo(X)} l\circ\hat{f}(\gamma)\cdot \chi_{\Gamma_1} (\gamma) \dis\pi (\gamma) + l\circ \reallywidehat{f^{-1}}(\gamma) \cdot \chi_{\Gamma_2}(\gamma)\dis\pi (\gamma)\\
&=&  \int_{\geo(X)} l(\gamma) \cdot (\chi_{\Gamma_1}+\chi_{\Gamma_2})(\gamma)\dis \pi (\gamma )=\int_{\geo(X)} l(\gamma) \dis \pi(\gamma).
\end{eqnarray*}
\end{proof}
\begin{rmk}\label{rmk:weaken}
The hypothesis in Lemma \ref{lm:02} can be we weaken. We may require the existence of the unique geodesic plan only for final measures satisfying $\mu_1\Lt\m$ rather than for an arbitrary $\mu_1 \in \p(X)$. We can repeat the proof choosing as final measure $$\mu_1:=\frac{1}{2}(\m(B_r(x))^{-1}\m|_{B_r(x)}+\m(f(B_r(x)))^{-1}\m|_{f(B_r(x))})$$ where $B_r(x)\subset \fix{f}^c$ is a sufficiently small ball.  
\end{rmk}
Consistently with Theorem \ref{th:Gi-Ra-St} we obtain
\begin{cor}\label{co:01}
Let $(X,\dis,\m)$ be an essentially non-branching $\mcp{K}{N}$-space and $f\in\iso$. If $\m(\fix{f})>0$ then $f=\mathbb{I}$.

In particular, this holds true for $\RCD$-spaces, essentially non-branching $\CD$-spaces, and essentially non-branching $\mathsf{CD}$-spaces.
\end{cor}
In order to use Gleason and Yamabe's characterization Theorem \ref{th:Gl-Ya} we need to show that $\G$ is a locally compact topological group. Recall that van Dantzig and van der Waerden have proved that $\iso$ is locally compact, granted that $X$ is locally compact and connected, see Theorem \ref{th:Da-Wa}. Below we prove that $\isom$ is a closed subgroup of $\iso$.
\begin{lemma}\label{cl:01}
Let $(X,\dis,\m)$ be a connected, locally compact $\mms$. Then $\isom$ is a locally compact closed subgroup of $\iso$ with respect to the compact-open topology.
\end{lemma}
\begin{proof}
We show that $\isom $ is closed. The local compactness of $\isom$ follows from the fact that $\isom$ is a closed subgroup of a locally compact group. Let $(f_n)_{n\in\mathbb{N}} \subset \isom$ be a converging sequence w.r.t. the compact-open topology with  limit $f:=\lim_{n \rightarrow \infty} f_n$. It is easy to see that $f $ is an isometry. Thus, to finish the proof, it remains to check that $(f)_{\#}\m=\m$. This follows from the regularity of the measure, as we argue below.

Indeed, since the measures $(f_n)_{\#} \m=\m$ are all equal, they trivially converge weakly to $\m$. On the other hand we will show that the pushforward of $\m$ under $f_n$ weakly converges to the measure $(f)_\# \m$. Therefore, $(f)_{\#}\m=\m$ by the uniqueness of the limit. By using the definition of the pushforward and the continuity of $g\circ f_n$, it is enough to verify that for every bounded continuous function with bounded support, $g:X\to\RR$, it holds that
\begin{eqnarray*}
\lim_{n\to\infty} \int_X g \circ f_n \; \mathrm{d}\m =  \int_X g \circ f \; \mathrm{d}\m,
\end{eqnarray*}
to show that  $(f_n)_{\#} \m  \overset{w}{\to}(f)_{\#}\m$. After the following observation it is clear that this last equality holds. 

Assuming that $g$ is as above we can construct an $\m$-integrable function, $G$, such that $|g_n(x)|\leq G(x)$ for all $x\in X$ and make use of the dominated convergence theorem. Take for example the multiple of the characteristic function $G:=k_g \,\chi|_{B_{r}(y)}$, where $k_g$ is a bound on $g$ and $r\in\mathbb{R}$ and $y\in X$ are such that $\cup_{n\in\mathbb{N}} \supp (g \circ f_n) \subset B_{r}(y)$. The existence of such a pair $\{r,y\}$ is guaranteed because $g$ has bounded support, and because $f_n \to f$ converges uniformly in compact subsets. The integrability of $G$ follows from $\m$ being finite on bounded sets. 
\end{proof}
We have now done all the work needed to prove Theorem \ref{th:03}. Compare Theorem \ref{th:03} to Theorem 4.5 in \cite{Ch-Co-II}.

\begin{proof}[Theorem \ref{th:03}]

Being the groups of isometries and of measure-preserving isometries locally compact spaces (Theorem \ref{th:Da-Wa}, and Lemma \ref{cl:01}) we can rely on Gleason and Yamabe's characterization of Lie groups. That is to say, $\G\in \{\iso,\isom\}$ is a Lie group if and only if $\G$ does not have the small subgroup property. Note that the contrapositive statements to Propositions  \ref{pr:01} and \ref{pr:bigfix implies ssp} show that $\G$ not having the $\ssp$ is equivalent to:

\textbf{($\mathsf{a}'$)}\textit{ There exist $x\in X$, $0<s$, $0<\xi<1$ such that for every non-trivial subgroup $\Lambda \subset \iso$ there exists an isometry $g\in \Lambda$ with }
\begin{equation*}
\m(\fix{g}\cap B_{s}(x))< \xi\; \m(B_{s}(x)).
\end{equation*}
As already observed in Remark \ref{re:01}, conditions \conditiona of Theorem \ref{th:03} and \textbf{($\mathsf{a}'$)} are equivalent. Indeed, it is clear that \conditiona implies \textbf{($\mathsf{a}'$)}. The other implication follows after observing that the existence of an isomorphism $\mathbb{I}\neq g\in\G$ with $\m(\fix{g}\cap B_{s}(x))\geq \xi\; \m(B_{s}(x))=\mathsf{Fix}$ implies that the measure of the fix point set of every element in the generated subgroup $\left<g \right>\neq \mathbb{I}$ is greater than or equal to $\mathsf{Fix}$. This proves the first part of the theorem. 
 
Finally, note that granted that $\iso$  has the no small subgroup property, then $\isom$ has the same property since they both are endowed with the compact-open topology. This shows that $\isom$ is a Lie group if $\iso$ is a Lie group.
\end{proof}
\begin{rmk}\label{rm:generalization}
The conclusion of Theorem \ref{th:03} remains valid under the following weaker assumption on tangent cones: Suppose that $X$ has $\m$-$a.e.$ unique tangent cones and that the set of all metric spaces that appear as unique tangent cones, $\mathsf{Tan}$, is compact. Furthermore, assume that there exist a constant $0<\mathsf{k}_0$ such that $\D{\mathsf{H^\infty}}{1}{y^\infty}>\mathsf{k}_0$ for every subgroup $\mathbb{I}\neq H^\infty\leq \mathrm{ISO}(Y^\infty)$ for all  $(Y^\infty,\dis_{Y^\infty},y^\infty)\in \mathsf{Tan}$. Indeed, under these assumptions the proof of Proposition \ref{pr:01} can be copied verbatim and this is the only part of the argument that depends on the hypothesis of the behavior of tangent cones.

This observation becomes relevant in the study of metric measure spaces which have well behaved tangents which might not be Euclidean. For example, any corank $1$ Carnot group of dimension $(k + 1)$ equipped with a left-invariant measure is an essentially non-branching $\MCP$-space with unique non-Euclidean tangents by Rizzi \cite{Ri}, thus by our results their automorphism groups are Lie groups. More generally, Le Donne proved in \cite{Do} that geodesic spaces equipped with a doubling measure that have $\m$-$a.e.$ unique tangents have  $\m$-$a.e.$ Carnot groups as tangents.

In hope of a clearer exposition we opted to simply make a remark and not to present Theorem \ref{th:03} in full generality since, as mentioned, the alternative proof does not contributes with new ideas. A complete exposition and discussion can be found in the author's thesis. 
\end{rmk}
\begin{rmk}\label{rmk:iso iff isom}
In general the implications
$$\iso \text{ is a Lie Group}\quad  (\impliedby)\implies \quad \isom \text{ is a Lie Group}.$$
need not hold in any direction. In Theorem \ref{th:03} the implication to the right side can be shown relying on regularity properties of the measure. Whilst the other direction is more drastic. There exist spaces, even with ``very'' regular measures, for which $\isom$ is a Lie group but $\iso$ is not.
\end{rmk}

\section{Metric measure spaces with $\isom$ not a Lie group}\label{se:ex}
We show that the $\MCP$-condition is not strong enough to guarantee that the group of measure-preserving isometries is a Lie group. We also see that for a  geodesic and compact $\mms$ with finite measure $\isom$ might fail to be a Lie group. We start by presenting a well-known example to develop intuition about the connection between $\isom$ not being a Lie group, $\isom$ having small subgroups, and the branching of geodesics.
\begin{example}\label{ex:01}
Denote the circle of radius $r$ by $S_r$. The Hawaiian earring, $\mathbb{H} $, is the space we obtain after gluing the circles $\{S_{\frac{1}{n^2}}\,\|\,n\in\N\}$ by identifying one point of every circle, see Figure \ref{fg:ex}. Endow $\mathbb{H}$ with the arc-length distance $\dis_\mathbb{H}$ and the 1-dimensional Hausdorff measure $\mathcal{H}^1$. This makes $(\mathbb{H},\dis_\mathbb{H},\mathcal{H}^1)$ a compact, geodesic metric measure space with finite measure. Observe that $\mathrm{ISO}(H) = \mathrm{ISO}_\m(H)= \Pi^\infty \{\pm 1 \} $ where the compact-open topology coincides with the product topology. Hence $\iso $ is totally disconnected but not discrete. By definition, $\iso$ is not a Lie group since  $\iso / \iso_0$  is not discrete.
\end{example}
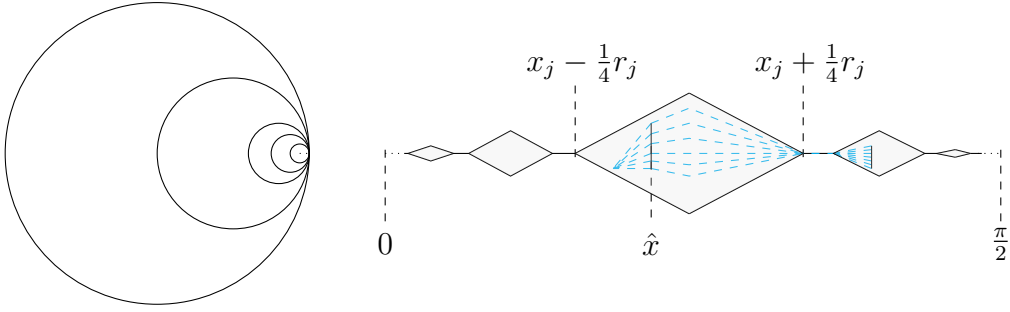
\begin{figure}
\begin{center}
\begin{tikzpicture}
\draw (2,2) circle (2);
\draw (3,2) circle (1);
\draw (3.6,2) circle (.4);
\draw (3.75,2) circle (.25cm);
\draw (3.875,2) circle (.125cm);
\draw[dotted] (3.875,2) -- (4,2);

\draw[dotted] (5,2) -- (5.25,2);
\draw (5.25,2) -- (5.3,2);
\draw (5.9,2)--(6.1,2) ;
\draw (7.2,2)--(7.5,2);
\draw (10.5,2) -- (10.9,2);
\draw (12.1,2) -- (12.26,2);
\draw (12.7,2) -- (12.85,2);
\draw[dotted] (12.8,2) -- (13.1,2);

\node at (5,.8) {0};
\node at (13.1,.8) {$\frac{\pi}{2}$};
\draw[thin][dashed] (5,1.1) -- (5,2.1);
\draw[thin][dashed] (13.1,1.1) -- (13.1,2.1);

\draw (5.3,2) -- (5.6,2.1)-- (5.9,2)-- (5.6,1.9)--cycle;
\fill[gray!20,nearly transparent]  (5.3,2) -- (5.6,2.1)-- (5.9,2)-- (5.6,1.9)--cycle;
\draw (6.1,2) -- (6.65,2.3)-- (7.2,2)-- (6.65,1.7)--cycle;
\fill[gray!20,nearly transparent]  (6.1,2) -- (6.65,2.3)-- (7.2,2)-- (6.65,1.7)--cycle;

\draw[thin] [dashed] (7.5,2.9) -- (7.5,1.9);
\node at (7.6,3.2) {$x_j - \frac{1}{4} r_{j}$};
\draw[thin] [dashed](10.5,2.9) -- (10.5,1.9);
\node at (10.6,3.2) {$x_j + \frac{1}{4} r_{j}$};

\draw[dashed][cyan] (8.5,2)--(8,1.8);
\draw[dashed][cyan] (8.5,2)--(10.5,2)--(10.9,2);
\draw[dashed][cyan] (11.4,2)--(10.9,2);
\draw[dashed][cyan] (8.5,2.13)--(8,1.8);
\draw[dashed][cyan](8.5,2.13)--(9,2.2)--(10.5,2);
\draw[dashed][cyan](11.4,1.8)--(10.9,2);
\draw[dashed][cyan] (8.5,2.26)--(8,1.8);
\draw[dashed][cyan](8.5,2.26)--(9,2.4)--(10.5,2);
\draw[dashed][cyan](11.4,1.9)--(10.9,2);
\draw[dashed][cyan] (8.5,2.4)--(8,1.8);
\draw[dashed][cyan](8.5,2.4)--(9,2.6)--(10.5,2);
\draw[dashed][cyan](11.4,1.95)--(10.9,2);
\draw[dashed][cyan] (8.5,1.9)--(8,1.8);
\draw[dashed][cyan](8.5,1.9)--(9,1.85)--(10.5,2);
\draw[dashed][cyan](11.4,1.85)--(10.9,2);
\draw[dashed][cyan] (8.5,1.8)--(8,1.8);
\draw[dashed][cyan](8.5,1.8)--(9,1.7)--(10.5,2);
\draw[dashed][cyan](11.4,2.1)--(10.9,2);
\draw[dashed][cyan](11.4,2.05)--(10.9,2);
\draw[black] (8.5,1.8) -- (8.5,2.4);
\draw[black] (11.4,1.8) -- (11.4,2.1);
\node at (8.5,.8) {$\hat{x}$};
\draw[thin] [dashed] (8.5,2.1) -- (8.5,1.1);

\draw (7.5,2) -- (9,2.8);
\draw (9,2.8) -- (10.5,2);
\draw (7.5,2) -- (9,1.2);
\draw (9,1.2) -- (10.5,2);
\fill[gray!20,nearly transparent] (7.5,2) -- (9,2.8) -- (10.5,2) -- (9,1.2) -- cycle;

\draw (10.9,2) -- (11.5,2.3) -- (12.1,2)-- (11.5,1.7) -- cycle;
\fill[gray!20,nearly transparent] (10.9,2) -- (11.5,2.3) -- (12.1,2)-- (11.5,1.7) -- cycle;

\draw (12.26,2) -- (12.5,2.05);
\draw (12.5,2.05) -- (12.7,2);
\draw (12.25,2) -- (12.5,1.95);
\draw (12.5,1.95) -- (12.7,2);
\fill[gray!20,nearly transparent] (12.26,2) -- (12.5,2.05) -- (12.7,2)-- (12.5,1.95) -- cycle;
\end{tikzpicture}
\caption{Hawaiian earring $\mathbb{H}$ and fancy necklace $\mathcal{FN}$. }
\label{fg:ex}
\end{center}
\end{figure}
\begin{example}\label{ex:02}
A \textit{fancy necklace} $(\mathcal{FN},\dis_\mathcal{FN},\m_\mathcal{FN})$ is a $\mGH$-limit of any sequence of $n$-diamonded $\mms$ $(\mathcal{N}^n,\dis_n,\m_n)$, called \text{$n$-necklaces}, which are inspired by a construction done by Ketterer and Rajala in \cite{Ke-Ra}. We define inductively the underlying sets $\mathcal{N}^n \subset \RR^2$. For this, first we write $I_k=[x_k-\frac{1}{4} \,r_k,x_k+\,\frac{1}{4} r_k]$ and define the diamond-shaped sets:
\begin{equation*}
D_{k}:=\left\{ (x,y) \in \mathbb{R}^2\;\; \| \;\; |y| \leq \frac{1}{9}(\frac{1}{4} \;r_k - |x-x_k|)\right \},
\end{equation*}
for $k\in\N$ and some sequence ${(r_n,x_n)}_{n\in \N}\subset \RR^2$ which we specify below.
We set $\mathcal{N}^0:=[0,\pi/2]\times \{0\} \subset \RR^2$ and define the $n$-necklace, $\mathcal{N}^n$, by replacing in $\mathcal{N}^{n-1}$ the segment $I_n \times \{0\}$ with $D_n$, for $n\in \N$, see Figure \ref{fg:ex}. To have a consistent construction we require that the sequence ${(r_n,x_n)}_{n\in \N}\subset \RR^2$ satisfies:
\begin{equation}\label{eq:rn}
\begin{aligned}
 &0<r_n\leq 1, \;\; \frac{1}{4}\,r_n\leq x_n \leq \frac{\pi}{2}-\frac{1}{4}\,r_n, \text{ and} & \\
&I_k\cap I_j = \emptyset  \text{ for } k<j.& 
\end{aligned}
\end{equation}
The last condition assures that different diamonds do not intersect. For $n\in \N \cup \{0\}$ endow $\mathcal{N}^n$ with the distance, $\dis_n=\dis_{L^{\infty}}$, induced from the  $L^\infty$-norm in $\RR^2$. To set a measure on the $n$-necklace we start by defining  $ \m_{D_n}\Lt \mathcal{L}^2$ on $D_n$ by
\begin{equation*}
\frac{\mathrm{d}\m_{D_n}}{\mathrm{d}\mathcal{L}^2}(x):=  \left[ \frac{2}{9}(\frac{1}{4}r_n - |x-x_n|)\right]^{-1} \cos^2(x)\, \chi|_{D_n}(x)\qquad \hfill \text{for } n\in \N,
\end{equation*}
and $\chi|_{A}$ the characteristic function of the set $A$. Denote by $\mathsf{D}^n =\cup_{1\leq k \leq n} D_k$, $\mathsf{L}^n:=\mathcal{N}^n \setminus \mathsf{D}^n$, and $L^0:=\mathcal{N}^0$. We set on $\mathcal{N}^n$ the measure $\m_{\mathcal{N}^n}$ defined as
\begin{eqnarray*}
\mathrm{d}\m_{\mathcal{N}^n}&:=&\mathrm{d}\m_{\mathsf{D}^n} + \cos^2(x) \,\mathrm{d}\mathcal{H}^1|_{\Ln}, \qquad \hfill \text{ where }\\
\m_{\mathsf{D}^n}&:=& \sum_{1\leq k\leq n }\m_{D_k}.
\end{eqnarray*}
In words, $\m_{\mathcal{N}^n}$ has a 2-dimensional contribution from $\Dn$ with constant density for fixed $x$-coordinate, and a 1-dimensional contribution from $\mathsf{L}^n$ absolutely continuous w.r.t. the 1-dimensional Hausdorff measure. Finally, we define the fancy necklace as the limit $(\mathcal{FN},\dis_\mathcal{FN},\m_\mathcal{FN}) := \mGH\-\lim_{n\to \infty} (\mathcal{N}^n,\dis_n,\m_n)$. The existence of the limit follows from the compactness of $\MCP$-spaces because Lemma \ref{lm:mcp} below shows that $n$-necklaces satisfy the $\mcp{2}{3}$-condition. 

We fix some notation before continuing. Given a sequence $\{(r_i,x_i)\}_{n\in\N}$, denote by $(P\mathcal{N}_k^{n-1},\dis_{n-1},\m'_{n-1})$ the projected $(n-1)$-necklace obtained from the sequence $\{(r_i,x_i)\}_{i\neq k}$ for $1\leq k\leq n$. That is, $P\mathcal{N}_k^{n-1}$ is the necklace with $n-1$ diamonds obtained by removing the $k$th-diamond from $\mathcal{N}^n$. We will write $x_n^\pm=x_n \pm 1/4 \, r_n$, and define the height as $h(w,B):=\mathcal H^1 (B \cap \{ x=w \})$ for $x\in \mathcal{N}^n$ and $B \subset \mathcal{N}^n$. Moreover, let
\begin{eqnarray*}
\Upsilon(B_0,B_1)&:= &\big\{\gamma \in \geo(\mathcal{N}^n) \,\|\, \gamma \text{ is a line segment with }\gamma_i\in B_i \text{, } i=0,1 \big\}.
\end{eqnarray*}
The set $\Upsilon(B_0,B_1)$ consists of Euclidean geodesics that go from $B_0$ to $B_1$. Lastly, for $|y|\leq r_k/36 $ and $k\in \N$ define $\gamma^{k,y}\in \geo(\mathcal{N}^n)$ as the geodesic obtained after gluing $\Upsilon((x_k^-,0),(x_k,y))$ with $\Upsilon((x_k,y),(x_k^+,0))$ and reparametrizing. The image of $\gamma^{k,y}$ is the union of a line segment going from the left vertex of $D_k$ to  $(x_k,y)$ with its reflection over $\{x=x_k\}$. Define $M^k$ as the set of all such geodesics.
\begin{lemma}\label{lm:mcp}
The $\mms$ $(\mathcal{FN},\dis_\mathcal{FN},\m_\mathcal{FN})$ satisfies the $\mcp{2}{3}$-condition.
\end{lemma}
\begin{proof}
The stability of the $\MCP$-condition assures that it's enough to show that $(\mathcal{N}^n,\dis_n,\m_n)$ $\in \mcp{2}{3}$ for every $n\in\N$ and every sequence $\{(x_k,r_k)\}\subset \RR^2$ that satisfies (\ref{eq:rn}), so we fix $n\in \N$ and such sequence. We proceed using key ideas from a proof in \cite{Ke-Ra}.

The conditions of Definition \ref{def:MCP} require that for every $\tilde{z}=(\tilde{x},\tilde{y})$ and $A \subset \mathcal{N}^n$ with $0 < \m_n(A) < \infty$ we give $\pi \in \mathcal{P} (\geo(\mathcal{N}^n))$ such that $(\e_0)_{\#} \pi=\delta_{\tilde{z}}$, $(\e_1)_\# \pi=\m_n(A)^{-1}\m_n|_{A}$, and inequality \ref{eq:mcp} is valid. 
Given $\tilde{z}$ and $A$ we will choose a set of geodesics $\Gamma=\Gamma_{\tilde{z},A} \subset \geo(\mathcal{N}^n)$ and define $\pi$ as the uniformly distributed probability measure over the set $\Gamma$. However, we reduce before the number of transports that need to be studied. 

To begin with, note that we can analyze separately  the sets  $A_{x'}= A \cap \{x=x'\}$ for a fixed $x'$. The simplification can be made because we will assure that the first coordinate contributes to the dilatation of the measure $\nt:=(\e_t)_\#\pi$ a factor equal to $t$, by picking geodesics with projection $\mathrm{p}_1(\gamma(t))= (1-t)\,\tilde{x} + t \,x' $ for $(x',y')=z'\in A$. Therefore the analysis reduces to estimating separately  the dilatation of the sets $A_{x'}$ for every $x'\in \mathrm{p}_1(A)$. Accordingly, to verify the $\mcp{2}{3}$-condition, it is enough to provide a set $\Gamma\subset \geo(\mathcal{N}^n)$ such that $\e_0(\Gamma)=\tilde{z}$, $\e_1(\Gamma)\in A_{x'}$, and
\begin{equation}\label{eq:mcp ex}  
\frac{\mathrm{d}\nt}{\mathrm{d}\m_n} (\gamma_t)\leq \frac{\sin^2 (l(\gamma))}{t \sin^2(t \, l (\gamma))} \frac{\mathrm{d}\mathfrak{n}_1}{\mathrm{d}\m_n} (\gamma_1) \qquad \text{for all } t\in[0,1],\, x' \in \mathrm{p}_1 (A), \, \gamma \in \Gamma.  
\end{equation}
\begin{clm*} It's sufficient to check that $(\mathcal{N}^m,\dis_m,\m_m) \in \mcp{2}{3}$ for $m=0,1,2$.
\end{clm*}
\begin{proof}\renewcommand{\qedsymbol}{$\blacksquare$}
First note that if $\tilde{z},z'\notin D_k$ for some $k\in\{1,...,n\}$ then we can choose $\Gamma$ in a way that makes the density of $\frac{\mathrm{d}\nt}{\mathrm{d}\m_n}$ independent of $y\in \mathrm{p}_2(D_k)$, that is, $\frac{\mathrm{d}\nt}{\mathrm{d}\m_n}((x,y))=\frac{\mathrm{d}\nt}{\mathrm{d}\m_n}((x,0))$ for $(x,y)\in D_k$. We can do this by choosing geodesics whose restriction to $D_k$ is exactly the set $M^k$. This choice of $\Gamma$ grants that the analysis of the transport of the measure inside $(\mathcal{N}^n,\dis_n,\m_n)$ is equivalent to the analysis of the transport of the measure inside the projected $(n-1)$-necklace $(P\mathcal{N}_k^{n-1},\dis_{n-1},\m'_{n-1})$. Furthermore, observe that if $n>2$ there exist at least $(n-2)$ such diamonds, $D_{k_i}$, for every $\tilde{z},z'\in \mathcal{N}^n$. Thus, for every transport inside $\mathcal{N}^n$ we can project at least $(n-2)$ times, reducing the task to checking the $\mcp{2}{3}$-condition in the $m$-necklaces, for $m=0,1,2$.
\end{proof}
In \cite{St-II} and \cite{Ke-Ra} it is shown that the $0$-necklace and $1$-necklace satisfy the $\mcp{2}{3}$-condition, this covers the cases of $m=0,1$ so we move to $m=2$.\footnote{More precisely, the proof in \cite{Ke-Ra} can be repeated verbatim by doing minor modifications.} We assume, because of symmetry, that $\tilde{x} \leq x$ and fix $\tilde{z} \in \mathcal{N}^2$ and $A_{x'}\subset \mathcal{N}^2$. We conclude from the preceding claim that the only situation left to check is that of  $\tilde{x}\in D_1$ and $x'\in D_2$. 
Let's first explain intuitively the way we transport the measure in this case. We start by expanding the measure uniformly from $\tilde{x}$ to a set $\hat{A}_{\hat{x}}$ with the same relative height as $A_{x'}$. Then we transport the measure from $\hat{A}_{\hat{x}}$ to $x_1^+$ without changing the relative height of the set $A_t:=\e_t(\Gamma)$ with respect to $D_1 \cap \{x=\gamma_t\}$. We continue through $\mathsf{L}^2$ and expand again keeping the heights ratio constant from $x_2^-$ to $A_{x'}$. The image of a transporting geodesic is the union of segments of straight lines described below, see Figure \ref{fg:ex}. In detail, to define $\Gamma$ first choose any set $\hat{A}_{\hat{x}}\subset D_1 \cap \{x=\hat{x}\}$ such that 
\begin{equation}\label{eq:h}
\frac{h(\hat{x},D_1)}{h(\hat{x},\hat{A}_{\hat{x}})} =\frac{h(x',D_2)}{h(x',A_{x'})},
\end{equation}
for $\hat{x} = \frac{1}{5}(\frac{r_1}{4}+4(\tilde{x}-x_1))+x_1$. Write $\hat{t}:= \frac{\hat{x} - \tilde{x}}{x-\tilde{x}}$, $t_1:= \frac{x_1^+ - \tilde{x}}{x-\tilde{x}}$, and $t_2:= \frac{x_2^- - \tilde{x}}{x-\tilde{x}}$ for the times at which the $x$-coordinate of any geodesic $\gamma \in \Upsilon(\tilde{x},A_{x'})$ is equal to $\hat{x}$, $x_1^+$, and $x_2^-$. Geodesics in $\Upsilon(\tilde{x},D_1 \cap \{x=\hat{x}\} )$ have the same length. Now define $\Gamma$ as the set of all geodesics satisfying the following:
$\rest{0}{\hat{t}}(\gamma)\in \Upsilon(\hat{x},\hat{A}_{\hat{x}})$,
 $\rest{\hat{t}}{t_1}(\gamma)\in M^1$,
$\rest{t_1}{t_2}(\gamma)\in \Upsilon(x_1^+,x_2^-)$, and 
$\rest{t_2}{1}(\gamma)\in M^2|_{A_{x'}}$, where $M^2|_{A_{x'}}$ is the subset of geodesics of $M^2$ that cross through $A_{x'}$. 

We estimate the density of the corresponding measure, for $\gamma (t)=(x_t,y_t)$ we have  that
\begin{equation*}\label{eq: est}
\frac{\mathrm{d}\nt}{\mathrm{d}\m}(\gamma_t) = \frac{1}{t} \frac{h(x_t,D_1)}{h(x_t,A_{t})} = \frac{1}{t^2} \frac{h(x_t,D_1)}{h(\hat{x},\hat{A}_{\hat{x}})} \frac{h(\hat{x},D_1)}{ h(\hat{x},D_1)} = \frac{1}{t^2} \frac{h(x_t,D_1)}{ h(x,D_1)} \frac{\mathrm{d}\mathfrak{n}_1}{\mathrm{d}\m}(\gamma_1),
\end{equation*}
for $0\leq t\leq \hat{t}$. The shape of the diamond $D_1$ allows to estimate   $\frac{h(x_t,D_1)}{ h(x,D_1)} \leq \left ( \frac{5}{4}-\frac{t}{4} \right)$. We can bound the time when the geodesics reach $\hat{x}$ by $\hat{t} \leq r_1 / 5 \leq 1/5$, and the length of the geodesics is necessarily $l\leq \pi/2$. Moreover in  \cite{Ke-Ra} the estimate $ \frac{5}{4}-\frac{t}{4} \leq t\, \frac{\sin^2(d)}{\sin^2(t\,d)}$ for all $(t,d) \in [0, 1/5]  \times  (0, \pi/2 + 1/4)$ is proved. Putting inequalities together we obtain inequality (\ref{eq:mcp ex}) for $t\in[0,\hat{t}]$. 

To finish, note that for $t\in [\hat{t},1]$, the relative density of $\nt$ is independent of the $y$-coordinate. Thus, its density is equal to the one of the transport in the $0$-necklace, which is a $\mathsf{MCP}(2,3)$-space. This shows that inequality (\ref{eq:mcp ex}) is satisfied also for $t\in [\hat{t},1]$, hence, in the complete interval $t\in [0,1]$.
\end{proof}
Observe that the automorphism groups of $(\mathcal{FN},\dis_\mathcal{FN},\m_\mathcal{FN})$ are  $\mathrm{ISO}(\mathcal{FN}) = \mathrm{ISO}_{\m}(\mathcal{FN}) = \Pi^\infty \{\pm 1 \}$. Thus we obtain Proposition $\ref{co:03}$, confirming that the measure contraction property is not strong enough to guarantee that the isometry group or the measure-preserving isometry group are Lie groups. 
\end{example}


\end{document}